\documentclass[12pt]{amsart}
\usepackage{a4wide,color,graphicx,amsmath,enumerate}
\allowdisplaybreaks

\let\pa\partial  
\let\na\nabla  
\let\eps\varepsilon  
\newcommand{\N}{{\mathbb N}}  
\newcommand{\R}{{\mathbb R}} 
\newcommand{\diver}{\operatorname{div}}

\newcommand{\E}{\mathbf{E}}
\newcommand{\s}{\mathbf{s}}
\newcommand{\m}{\mathbf{m}}
\renewcommand{\H}{\mathbf{H}}
\renewcommand{\u}{\mathbf{u}}
\newcommand{\Je}{\mathbf{J}_e}
\newcommand{\Js}{\mathbf{J}_s}
\newcommand{\bnu}{\boldsymbol\nu}
\newcommand{\curl}{\operatorname{curl}}

\newtheorem{theorem}{Theorem}   
\newtheorem{lemma}[theorem]{Lemma}   
\newtheorem{proposition}[theorem]{Proposition}   
\newtheorem{remark}[theorem]{Remark}   
\newtheorem{corollary}[theorem]{Corollary}

\begin{document}  

\title[A spin drift-diffusion Maxwell-Landau-Lifshitz system]{
Analysis of a coupled spin drift-diffusion
Maxwell-Landau-Lifshitz system}

\author[N. Zamponi]{Nicola Zamponi}
\address{Institute for Analysis and Scientific Computing, Vienna University of  
Technology, Wiedner Hauptstra\ss e 8--10, 1040 Wien, Austria}
\email{nicola.zamponi@tuwien.ac.at}

\author[A. J\"{u}ngel]{Ansgar J\"{u}ngel}
\address{Institute for Analysis and Scientific Computing, Vienna University of  
Technology, Wiedner Hauptstra\ss e 8--10, 1040 Wien, Austria}
\email{juengel@tuwien.ac.at}

\date{\today}

\thanks{The authors acknowledge partial support from   
the Austrian Science Fund (FWF), grants P24304, P27352, and W1245}  

\begin{abstract}
The existence of global weak solutions to a coupled spin drift-diffusion 
and Maxwell-Landau-Lifshitz system is proved. The equations are considered
in a two-di\-men\-sional magnetic layer structure and are supplemented with 
Dirichlet-Neumann boundary conditions. The spin drift-diffusion model 
for the charge density and spin density vector
is the diffusion limit of a spinorial Boltzmann equation for a vanishing spin
polarization constant. The Maxwell-Landau-Lifshitz system consists of the
time-dependent Maxwell equations for the electric and magnetic fields and
of the Landau-Lifshitz-Gilbert equation for the local magnetization,
involving the interaction between magnetization and spin density vector.
The existence proof is based on a regularization procedure, $L^2$-type estimates,
and Moser-type iterations which yield the 
boundedness of the charge and spin densities.
Furthermore, the free energy is shown to be nonincreasing in time
if the magnetization-spin interaction constant in the Ladau-Lifshitz equation is 
sufficiently small.
\end{abstract}

\keywords{Spin drift-diffusion equations, Maxwell-Landau-Lifshitz system,
existence of weak solutions, von-Neumann entropy, bounded weak solutions.} 

\subjclass[2010]{35K51, 35Q61, 35Q60, 82D40.}  

\maketitle


\section{Introduction}\label{sec.intro}

Magnetic devices, such as magnetic sensors and hard disk read heads,
typically consist of ferromagnetic/nonmagnetic layer structures. A model
for magnetic multi-layers was first introduced by Slonczewski \cite{Slo96}.
This model is well suited for Magnetoresistive Random Access Memory (MRAM) devices
but it is less appropriate for current-driven domain wall-motion. A more general
approach is to introduce the spin accumulation coupled to the magnetization
dynamics. The evolution of the magnetization is modeled by the 
Landau-Lifshitz (-Gilbert) equation \cite{ZLF02}. 
When electrodynamic effects cannot be neglected (like in high-frequency regimes), 
this description needs to be coupled to the Maxwell equations. 
In this paper, we analyze for the first time a coupled spin drift-diffusion 
Maxwell-Landau-Lifshitz system in two space dimensions with physically
motivated boundary conditions.

Let us describe our model in more detail. We consider a three-layer semiconductor 
structure $\Omega\subset\R^2$ consisting
of two ferromagnetic regions $\omega_1$, $\omega_2\subset\Omega$, separated
by a nonmagnetic interlayer $\Omega\backslash\omega$, where 
$\omega=\omega_1\cup\omega_2$ is the union of magnetic layers \cite{AHPP14}.

\smallskip
{\bf Landau-Lifshitz-Gilbert equation.}
The dynamics of the magnetization $\m=(m_1,$ $m_2,m_3)$
is governed by the Landau-Lifshitz-Gilbert (LLG) equation 
\begin{equation}\label{1.m}
  \pa_t\m = \m\times(\Delta\m+\H+\beta\s) 
	- \alpha\m\times(\m\times(\Delta\m+\H+\beta\s))
	\quad\mbox{in }\omega,\ t>0,
\end{equation}
where the effective field $\H_{\rm eff}=\Delta\m+\H$ consists of the sum of the
exchange field contribution $\Delta\m$ and the magnetic field $\H$, and
$\alpha>0$ denotes the Gilbert damping constant.
The additional term $\beta\s$ models the interaction between
the magnetization $\m$ and spin accumulation $\s$ with strength $\beta>0$
\cite{Cim07,ZLF02}. We choose the initial and boundary conditions
\begin{equation}\label{1.ibc.m}
  \m(0)=\m^0\quad\mbox{in }\omega, \quad
	\na\m\cdot\bnu = 0\quad\mbox{on  }\omega,\ t>0,
\end{equation}
where $\bnu$ is the outward unit normal on $\pa\omega$, we write $\m(0)=\m(\cdot,0)$, 
and the notation $\na\m\cdot\bnu=0$ means that $\na m_i\cdot\bnu=0$ for $i=1,2,3$. 
The Neumann conditions were also used in, e.g., \cite{AHPP14,GaWa07}. 
We set $\m=0$ in $\Omega\backslash\omega$.

The existence and non-uniqueness of weak solutions to the LLG equation goes back to 
\cite{AlSo92,Vis85}. The local existence of a unique strong solution was proven
in \cite{CaFa01}. In two space dimensions and for sufficiently small initial data,
the strong solution is, in fact, global in time \cite{CaFa01}. 
For general initial data, the two-dimensional solution may develop finitely
many point singularities after finite time; see \cite{Har04} for a discussion.
The existence of weak solutions in three space dimensions with physically 
motivated boundary
conditions was shown in \cite{BFFG14}, based on a finite-element approximation.
For a complete review on analytical results, we refer to \cite{Cim08,KrPr06}.

\smallskip
{\bf Maxwell equations.} The Maxwell equations are given by the time-dependent
Amp\`ere and Faraday laws for the electric and magnetic fields $\E=(E_1,E_2,E_3)$ 
and $\H=(H_1,H_2,$ $H_3)$, respectively,
\begin{equation}
  \pa_t\E - \curl\H = \Je, \quad \pa_t\H + \curl\E = -\pa_t\m 
	\quad\mbox{in }\Omega,\ t>0, \label{1.M1}
\end{equation}
and by the Gauss laws
\begin{equation}
  \diver\E = \rho-C(x), \quad \diver(\H+\m) = 0 \quad\mbox{in }\Omega,\ t>0.
	\label{1.M2}
\end{equation}
Here, $\Je$ is the electron current density, $\rho$ the electron charge density,
and $C(x)$ the doping concentration characterizing the device under
consideration. We assume that the boundary $\pa\Omega$ splits into two parts:
the Ohmic contacts $\Gamma_D$ and the union $\Gamma_N$ of the insulating parts,
with $\pa\Omega=\Gamma_D\cup\Gamma_N$. Then the initial and boundary conditions
of $\E$ and $\H$ are given by
\begin{align}
  & \E(0) = \E^0, \quad \H(0) = \H^0\quad\mbox{in }\Omega, 
	\label{1.ic.M} \\
  & \E\times\bnu = 0 \quad\mbox{on }\Gamma_D,\ t>0, \quad
	\H\times\bnu = 0 \quad\mbox{on }\Gamma_N,\ t>0,	\label{1.bc.M1} \\
	& \E\cdot\bnu = 0 \quad\mbox{on }\Gamma_N,\ t>0. \label{1.bc.M2}
\end{align}
The existence analysis (for given
and smooth $\Je$ and $\pa_t\m$) may be based on Kato's theory of 
quasilinear evolution equations of hyperbolic type \cite{Mil82} or on
semigroup theory \cite{Joc96}; also see Section \ref{sec.max}.

Coupled Maxwell and LLG equations were intensively studied in the literature.
For instance, the Maxwell-Landau-Lifshitz system in three space dimensions 
with periodic boundary conditions 
was investigated in \cite{GuSu97}. Carbou and Fabrie \cite{CaFa01a} proved the
existence of weak solutions to the LLG equation, coupled to Maxwell's equations,
in the whole space $\R^3$. The existence of spatially periodic strong solutions 
in three dimensions and their local uniqueness were proved in \cite{Cim07}. 
The solutions are only partially regular (i.e.\ smooth except on a low-dimensional
set) because of possible vortices or phase transitions. We refer to
\cite{DiLi09} for the two-dimensional case and to \cite{DiGu08,DLW09} for three 
space dimensions.

\smallskip
{\bf Spin drift-diffusion system.} We consider the spin drift-diffusion equations
for the charge density $\rho$ and the spin density vector $\s=(s_1,s_2,s_3)$
\begin{align}
  & \pa_t\rho - \diver \Je = 0, \quad \Je = D(\na\rho - \rho\E), \label{1.rho} \\
	& \pa_t\s - \diver\Js + \gamma \m\times\s = -\s/\tau, \quad
	\Js =  D(\na\s - \s\otimes\E), \label{1.s}
\end{align}
where $D>0$ is the diffusivity constant, $\Js$ the spin current density vector,
$\gamma>0$ is the strength of the effective magnetic field, and $\tau>0$
denotes the spin-flip relaxation time. The term $\gamma\m\times\s$ causes the
spin density vector to rotate around the magnetization, while the spin-flip
relaxation term leads, in the absence of other forces, to exponential decay
to the equilibrium spin density vector $\s_{\rm eq}=0$. 
We assume that the densities $\rho$ and $\s$ are prescribed 
on $\Gamma_D$ (Ohmic contacts),
while there are no-flux boundary conditions on $\Gamma_N$
(insulating boundary). This results in the initial and boundary conditions
\begin{align}
  \rho(0) = \rho^0, \quad \s(0) = \s^0 &\quad\mbox{in }\Omega, 
	\label{1.ic.rho} \\
  \rho = \rho_D, \quad \s = 0 &\quad \mbox{on }\Gamma_D,\ t>0, \label{1.bc.rhoD} \\
	\Je\cdot\bnu = 0, \ \Js\cdot\bnu = 0 &\quad\mbox{on }\Gamma_N,\ t>0.
	\label{1.bc.rhoN}
\end{align}
The spin current density is a $3\times 3$ matrix with rows
${\mathbf J}_{s,i}=\na s_i-s_i\E$ for $i=1,2,3$. Accordingly,
$\Js\cdot\bnu$ is a vector in $\R^3$ consisting of the elements 
${\mathbf J}_{s,i}\cdot\bnu=0$, $i=1,2,3$.

Spin-polarized drift-diffusion models were analyzed only recently in the
literature. Glitz\-ky \cite{Gli08} proves the existence and uniqueness of weak 
solutions to a two-dimensional transient drift-diffusion system for 
spin-up and spin-down densities.
The stationary problem was solved in three space dimensions in \cite{GaGl10}.
These models were derived from the spinor Boltzmann equation in the diffusion
limit with strong spin-orbit coupling in \cite{ElH14}.

More detailed information can be obtained by introducing the spin density.
Spin-vector drift-diffusion equations can be derived from the 
spinor Boltzmann equation by assuming 
a moderate spin-orbit coupling \cite{ElH14}. Projecting the
spin-vector density in the direction of the magnetization, we recover
the two-component drift-diffusion system as a special case. 
In \cite{ElH14}, the scattering rates are supposed to be scalar quantities. 
Assuming that the scattering rates are positive definite Hermitian matrices,
a more general matrix drift-diffusion model was derived in \cite{PoNe11}.
The global existence of weak solutions to this model was shown in \cite{JNS15}.
An energy-dissipative finite-volume discretization was presented in \cite{CJS15}.

Equations \eqref{1.rho}-\eqref{1.s} result from the cross-diffusion
model in \cite{PoNe11} by choosing a vanishing spin polarization constant. 
By this choice, the diffusion matrix becomes diagonal which makes our analysis
possible. For a more general
spin drift-diffusion LLG model, but without coupling to Maxwell's equations and with 
saturating drift velocity, we refer to \cite{Zam14}.

In the physical literature, also other equations for the spin density vector have been
suggested. In \cite[Formula (8)]{Mia08}, the spin density is defined as the
difference of the spin-up and spin-down densities. Thus, the underlying equation
is a two-component model which is a special case of the general model.
Starting from kinetic equations for the charge and spin components
of the Wigner-transformed density matrix, Lueffe et al.\ \cite[Formula (54)]{LKN11}
derived a spin diffusion equation for weak spin-orbit interaction or strong 
scattering.
Another derivation employs a $SU(2)$ gauge field theoretical description of the
spin-orbit coupling and the Heisenberg field operators for the definition
of the spin density \cite[Formulas (1)-(4)]{SRV14}. The resulting equation
is similar to \eqref{1.s} but the spin current density also depends on the
charge current.
Finally, assuming that the diffusivity in the drift-diffusion equation for
the density matrix is proportional to the magnetization vector, the authors
in \cite{ZLF02} obtain \eqref{1.s} with a spin current density whose drift
term equals $\m\otimes\E$ instead of $\s\otimes\E$ as in \eqref{1.s}.
The former drift term can be derived from the Wigner equation in the diffusion
limit by approximating the Wigner function appropriately \cite[Formula (23)]{CGY15}.
We stress the fact that the model \eqref{1.rho}-\eqref{1.s} is derived from
the spinor Boltzmann equation without heuristic arguments. 

\smallskip
{\bf Main results.}
We show that there exists a global-in-time weak solution to the coupled
spin drift-diffusion Maxwell-LLG system. Our assumptions are as follows:
\begin{align}
  & \omega\subset\Omega\subset\R^2\mbox{ are bounded domains with smooth boundaries},
	\label{hypo.dom} \\
  & \alpha,\,\beta,\,\gamma,\,D,\,\tau>0, \quad C\in L^\infty(\Omega),
	\label{hypo.CD} \\
	& \rho_D,\ \rho^0,\ \s^0,\ \E^0,\ \H^0\in H^1(\Omega), \quad
	\m^0\in H^1(\omega), \quad |\m^0|=1\ \mbox{in }\omega, \label{hypo.bc} \\
	&   \diver\E^0 = \rho^0-C(x), \quad \diver(\H^0+\m^0)=0\ \mbox{in }\Omega, \quad
	\E^0\cdot\bnu = 0\ \mbox{on }\Gamma_N. \label{hypo.ic}
\end{align} 
We also suppose that $\pa\Omega=\Gamma_D\cup\Gamma_N$, 
$\Gamma_D\cap\Gamma_N=\emptyset$, and $\Gamma_N$ is open and has positive
measure in $\pa\Omega$. To simplify the notation, we write sometimes
$\mathbf{u}\in H^1(\Omega)$ instead of $\u\in H^1(\Omega)^3$ for vector-valued
functions $\u$. We denote by $H_D^1(\Omega)$ the space of
all functions in $H^1(\Omega)$ with zero trace on $\Gamma_D$ and by
$H_D^1(\Omega)'$ its dual space.

Let us discuss assumptions \eqref{hypo.dom}-\eqref{hypo.ic}. 
The restriction to two space dimensions 
is (only) needed in the uniqueness proof for the regularized LLG equation 
\eqref{2.m}. This property is required to obtain a well-defined fixed-point operator.
In \eqref{hypo.ic}, we suppose that equations \eqref{1.M2} and 
\eqref{1.bc.M2} hold initially. These properties allow us to conclude the 
validity of \eqref{1.M2} and \eqref{1.bc.M2} from \eqref{1.M1} and \eqref{1.bc.M1}
(see e.g.\ \cite[p.~435f.]{DaLi90}).

The {\em first main result} is the following theorem.

\begin{theorem}[Existence of global weak solutions]\label{thm.ex}
Let assumptions \eqref{hypo.dom}-\eqref{hypo.ic} hold.
Then there exists a weak solution to \eqref{1.m}-\eqref{1.bc.rhoN} satisfying
\begin{align*}
  & \rho\ge 0\quad\mbox{in }\Omega, \quad 
	\m=0\quad\mbox{in }\Omega\backslash\omega,\ t>0, \\
  & \rho,\ \s\in L^2_{\rm loc}(0,\infty;H^1(\Omega))\cap 
	L^\infty_{\rm loc}(0,\infty;L^\infty(\Omega)), \quad
	\pa_t\rho,\ \pa_t\s\in L^2_{\rm loc}(0,\infty;H_D^1(\Omega)'), \\
	& \E,\ \H\in C^0([0,\infty);L^2(\Omega)), \\
	& \m\in L^\infty_{\rm loc}(0,\infty;H^1(\omega)), \quad
	\pa_t\m\in L^2_{\rm loc}(0,\infty;L^2(\omega)), \quad |\m|=1\mbox{ in }\omega.
\end{align*}
\end{theorem}

The $L^\infty$ bounds on $\rho$ and $\s$ can be shown to be 
uniform in time; see Remark \ref{rem.Linfty}.

The proof of this theorem is based on a combination of semigroup techniques
for the Maxwell equations \eqref{1.M1}, a Galerkin approximation for the
LLG equation \eqref{1.m}, and $L^2$ estimates for the spin drift-diffusion
model \eqref{1.rho}-\eqref{1.s}. Note that it is sufficient to solve
\eqref{1.M1} with \eqref{1.ic.M}-\eqref{1.bc.M1} as \eqref{1.M2} with \eqref{1.bc.M2}
are consequences of the former equations.
Since $\Je$ and $\pa_t\m$ are not regular
a priori, we approximate these terms by regularizing $\na\rho$, $\pa_t\m$
and truncating $\rho$, $\s$ in the drift terms in \eqref{1.rho}-\eqref{1.s}, 
respectively. This regularization is
similar to that employed by Jochmann \cite{Joc96} for a coupled Maxwell
drift-diffusion system (without spin). The challenge in the proof is to remove the
regularization and truncation. For the de-regularization limit, 
we derive uniform estimates for the variables by showing that the functional
\begin{equation}\label{1.S}
  S(t) = \frac12\int_\Omega\big((\rho-\rho_D)^2 + |\s|^2 + |\E|^2 + |\H|^2\big)dx
	+ \frac12\int_\omega|\na\m|^2 dx
\end{equation}
satisfies the inequality
$$
  S(t) + c_1\int_0^t\int_\Omega\big(|\na\rho|^2 + |\na\s|^2\big)dx
	+ c_2\int_0^t\int_\omega|\pa_t\m|^2 dx \le c_3(T), \quad t\in(0,T),
$$
where $c_1$, $c_2$, $c_3(T)>0$ are some constants which are independent of the
solution. Further details on the proof are given in Section \ref{sec.strat}.
In order to remove the truncation, we derive $L^\infty$ estimates for $\m$, 
$\rho$, and $\s$ by using a Moser-type iteration procedure.

The functional $S(t)$ is not the energy of the system. The (relative)
free energy consists of the von-Neumann energy for the spin system, 
the electromagnetic energy, and the exchange energy of the magnetization:
\begin{align}
  E(t) &= \frac12\int_\Omega\big(\rho_+(\log\rho_+-1) + \rho_-(\log\rho_--1)
	- 2\log\rho_D(\rho-\rho_D)\big)dx \label{1.E} \\
	&\phantom{xx}{}+ \frac12\int_\Omega\big(|\E-\log\rho_D|^2 + |\H|^2\big)dx
	+ \frac12\int_\omega|\na\m|^2 dx, \nonumber
\end{align}
where $\rho_\pm=\rho\pm|\s|$ (see Section \ref{sec.en}). 
This formulation implicitly assumes that $\rho\ge|\s|$. 
Our second main result is the proof that $E(t)$ is nonincreasing in time
under the conditions that the interaction parameter $\beta>0$ is sufficiently
small and the solution is smooth and satisfies $\rho>|\s|$. 
This shows that the coupled system dissipates the free energy.
The constraint on the parameter
$\beta$ may come from the fact that the term $\beta\s$ is introduced
in the LLG equation only heuristically, but we leave further investigations
to future research.

The paper is organized as follows. The strategy of the existence proof is
explained in Section \ref{sec.strat} and the full proof is given in Section
\ref{sec.ex}. We conclude in Section \ref{sec.en} with the monotonicity 
proof for the free energy $E(t)$.


\section{Strategy of the proof of Theorem \ref{thm.ex}}\label{sec.strat}

In order to prove Theorem \ref{thm.ex}, we first consider a
truncated and regularized problem. For this, let $T>0$, $\eps>0$, $M>0$ and set
$[z]_M:=\min\{M,\max\{0,z\}\}$ for $z\in\R$. We wish to prove the existence of weak
solutions to 
\begin{align}
  & \pa_t\rho - \diver(D(\na\rho - [\rho]_M\E)) = 0, \label{2.rho} \\
	& \pa_t\s - \diver\left(D\left(\na\s - [|\s|]_M\frac{\s}{|\s|}\otimes\E\right)\right)
	+ \gamma\m\times[|\s|]_M\frac{\s}{|\s|} = -\frac{\s}{\tau}, \label{2.s} \\
	& \pa_t\E - \curl\H = D(\na R_\eps^x(\rho)-[\rho]_M\E), \label{2.E} \\
	& \pa_t\H + \curl\E = -\pa_t R_\eps^t(\m)\quad\mbox{in }\Omega\times(0,T),
	\label{2.H}, \\
	& \pa_t\m - \eps\Delta\m = \m\times(\Delta\m+\H+\beta\s) 
	- \alpha\m\times(\m\times(\Delta\m+\H+\beta\s))\ \mbox{in }\omega\times(0,T)
	\label{2.m}
\end{align}
with the initial and boundary conditions \eqref{1.ibc.m},
\eqref{1.ic.M}-\eqref{1.bc.M1}, and \eqref{1.ic.rho}-\eqref{1.bc.rhoN}. 
In the Maxwell equations \eqref{2.E}-\eqref{2.H}, $R_\eps^x$ and $R_\eps^t$
are two families of linear regularization operators acting on functions of $x$ and
$t$, respectively, satisfying for all $u\in L^2(\Omega)$ and $v\in L^2(0,T)$,
\begin{align}
  & \|R_\eps^x(u)\|_{C^1(\overline\Omega)}\le k_\eps\|u\|_{L^2(\Omega)}, 
	\label{2.Rx1} \\	
	& \|R_\eps^x(u)\|_{H^1(\Omega)}\le k_0\|u\|_{H^1(\Omega)}, &&
	\lim_{\eps\to 0}\|R_\eps^x(u)-u\|_{L^2(\Omega)} = 0, \label{2.Rx2} \\
	& \|R_\eps^t(v)\|_{C^1([0,T])}\le k_\eps\|v\|_{L^2(0,T)}, 
	\label{2.Rt1} \\
	& \|R_\eps^t(v)\|_{H^1(0,T)}\le k_0\|v\|_{H^1(0,T)}, &&
	\lim_{\eps\to 0}\|R_\eps^t(v)-v\|_{L^2(0,T)} = 0, \label{2.Rt2}
\end{align}
where $k_\eps>0$ depends on $\eps$ but $k_0>0$ is independent of $\eps$.
The space-regularization operator $R_\eps^x$ was introduced in \cite[p.~665f]{Joc96},
where also their existence and the above properties were proved. The 
time-regularization operator $R_\eps^t$ can be defined in a similar way.

In the following, we abbreviate 
$X=L^2(0,T;L^2(\Omega))^4$ and $Y=C^0([0,T];L^2(\Omega))^3$.

The {\em first step} of the proof of Theorem \ref{thm.ex} is the application of
the Leray-Schauder fixed-point theorem to the map
\begin{equation*}
  F:X\times Y\times[0,1] \to X\times Y, \quad
	(\rho,\s;\m;\sigma) \mapsto (\rho^*,\s^*;\m^*),
\end{equation*}
which is defined as follows (details will be given in the following subsections).
Let $(\rho,\s;\m;$ $\sigma)\in X\times Y\times [0,1]$ be given. 
\begin{enumerate}[{\bf I.}]
\item Solve the regularized Maxwell equations
\begin{align}
  \pa_t\E - \curl\H &= \sigma D(\na R_\eps^x(\rho)-[\rho]_M\E), \label{fix.M1} \\
	\pa_t\H + \curl\E &= -\sigma\pa_t R_\eps^t(\m)\quad\mbox{in }\Omega\times(0,T),
	\label{fix.M2}
\end{align}
with initial and boundary conditions 
\begin{align}
  & \E(0) = \sigma\E^0, \quad \H(0) = \sigma\H^0\quad\mbox{in }\Omega, 
	\label{fix.ic.M}\\
  & \E\times\bnu = 0 \quad\mbox{on }\Gamma_D,\ t>0, \quad
	\H\times\bnu = 0 \quad\mbox{on }\Gamma_N,\ t>0, \label{fix.bc.M}
\end{align}
and obtain $(\E,\H)\in C^0([0,T];L^2(\Omega))$.
\item Solve the regularized (nonlinear) LLG equations
\begin{align}
  \pa_t\m^* - \eps\Delta\m^* &= \m^*\times(\Delta\m^*+\sigma\H+\sigma\beta\s) 
	\label{fix.m} \\
	&\phantom{xx}{}- \alpha\m^*\times(\m^*\times(\Delta\m^*+\sigma\H+\sigma\beta\s))
	\quad\mbox{in }\omega\times(0,T) \nonumber
\end{align}
with initial and boundary conditions
\begin{equation}\label{fix.bc.m}
  \m^*(0)=\sigma\m^0\quad\mbox{in }\omega, \quad
	\na\m^*\cdot\bnu = 0\quad\mbox{on  }\omega,\ t>0,
\end{equation}
and obtain
$\m^*\in L^\infty(0,T;H^1(\Omega))\cap H^1(0,T;L^2(\Omega))$ satisfying
$|\m^*|\le 1$ in $\Omega\times(0,T)$.
\item Solve the linearized spin drift-diffusion equations
\begin{align}
  & \pa_t\rho^* - \diver(D(\na\rho^* - \sigma[\rho^*]_M\E)) = 0
	\quad\mbox{in }\Omega\times(0,T), \label{fix.rho} \\
	& \pa_t\s^* - \diver\left(D\left(\na\s^* 
	- \sigma[|\s^*|]_M\frac{\s^*}{|\s^*|}\otimes\E\right)\right)
	+ \sigma\gamma\m^*\times[|\s^*|]_M\frac{\s^*}{|\s^*|} 
	= -\frac{\s^*}{\tau}, \label{fix.s}
\end{align}
with the initial and boundary conditions 
\begin{align*}
  \rho^*(0) = \sigma\rho^0, \quad \s^*(0) = \sigma\s^0
	&\quad\mbox{in }\Omega, \\
	\rho^*=\sigma\rho_D, \quad \s^*=0 &\quad\mbox{on }\Gamma_D,\ t>0, \\
	(\na\rho^* - \sigma[\rho^*]_M\E)\cdot\bnu 
	= D\left(\na\s^* - \sigma[|\s^*|]_M\frac{\s^*}{|\s^*|}\otimes\E\right)\cdot\bnu = 0
	&\quad\mbox{on }\Gamma_D,\ t>0,
\end{align*}
and obtain $(\rho^*,\s^*)\in L^2(0,T;H^1(\Omega))\cap H^1(0,T;H_D^1(\Omega)')
\subset X$.
\end{enumerate}

The regularization in \eqref{2.m} ensures that the solution is unique, which
is necessary for the definition of the fixed-point operator. Observe that
$F(\rho,\s;\m;0)=(0,0;0)$ since the solutions to the homogeneous subproblems
($\sigma=0$) are trivial. Standard arguments show that the operator $F$ is continuous.
By Aubin's lemma \cite{Sim87}, bounded sequences in 
$L^2(0,T;H^1(\Omega))\cap H^1(0,T;H_D^1(\Omega)')$ are relatively compact in
$L^2(0,T;L^2(\Omega))$ and bounded sequences in 
$L^\infty(0,T;H^1(\Omega))\cap H^1(0,T;L^2(\Omega))$ are relatively compact in
$C^0([0,T];L^2(\Omega))$. Consequently, $F$ is compact. 
It remains to prove uniform estimates for all fixed points of $F(\cdot,\sigma)$.
They will be derived from estimates for the functional $S(t)$ defined
in \eqref{1.S}; see Section \ref{sec.est}. 
Then the Leray-Schauder fixed-point theorem
implies the existence of a fixed point of $F(\cdot,1)$, i.e.\ of a solution
to \eqref{2.rho}-\eqref{2.m} with the corresponding initial and boundary conditions.

The estimates from $S(t)$ turn out to be independent of $\eps$ which allows us
in the {\em second step} of the proof to perform the limit $\eps\to 0$.
The proof that we can remove the truncation in \eqref{2.rho}-\eqref{2.s}
is more delicate. We prove in the {\em third step} $L^\infty$ bounds for 
$\rho$ and $\s$ by employing a Moser-type iteration technique. The idea is
to derive $L^p$ estimates of $\rho$ and $\s$, which are independent of $p$,
and then to pass to the limit $p\to\infty$. By a refined Moser-Alikakos
iteration technique, it is even possible to show that the $L^\infty$ bounds
are uniform in time; see Remark \ref{rem.Linfty}.

The proof of $|\m|=1$ is slightly different.
First, we show for the regularized LLG equation \eqref{2.m}, by using a Moser-type
iteration, that $\|\m\|_{L^\infty(\omega)}\le 1$. After the limit $\eps\to 0$,
we can take the inner product of the limit equation \eqref{1.m} and $\m$
to deduce immediately that $|\m|=1$ in $\omega$, $t>0$.


\section{Proof of Theorem \ref{thm.ex}}\label{sec.ex}

\subsection{Analysis of the regularized Maxwell equations}\label{sec.max}

We show that the regularized Maxwell equations \eqref{2.E}-\eqref{2.H} are
uniquely solvable.

\begin{lemma}[Existence of the regularized Maxwell equations]\label{lem.max}
Let $(\rho,\m)\in L^2(0,T;L^2(\Omega))$ $\times C^0([0,T];L^2(\Omega))^3$ and
$\sigma\in[0,1]$. Then there exists a unique mild solution 
$(\E,\H)\in C^0([0,T;$ $L^2(\Omega))^6$ to \eqref{fix.M1}-\eqref{fix.bc.M}.
\end{lemma}

Note that equations \eqref{1.M2} and \eqref{1.bc.M2} are yet not proved.
They will be shown in Section \ref{sec.eps} to hold for the de-regularized system.

\begin{proof}
The proof is based on semigroup theory and the Banach fixed-point theorem.
In principle, a fixed-point argument is not necessary since the Maxwell
equations are linear. However, we would need to deal with a non-autonomous operator
because of the presence of the term $[\rho(x,t)]_M\E$ on the right-hand side
of \eqref{fix.M1}. Therefore, we prefer the simple fixed-point argument.
Following \cite{Joc96}, we introduce the spaces
\begin{align*}
  W &= \big\{\u \in L^2(\Omega)^3: \curl\u \in L^2(\Omega)^3\big\}, \\
	W_E &= \bigg\{\u \in W: \int_\Omega(\boldsymbol\phi\cdot\curl\u 
	 - \u\cdot \curl\boldsymbol\phi)dx = 0\mbox{ for }\boldsymbol\phi
	\in C_0^\infty(\R^2\backslash\overline{\Gamma}_N)^3\bigg\}, \\
	W_H &= \bigg\{\mathbf{v}\in W: \int_\Omega(\u \cdot\curl\mathbf{v}
	- \mathbf{v}\cdot\curl\u )dx = 0 \mbox{ for }\u \in W_E\bigg\}.
\end{align*}
The space $W_E$ consists of all functions $\u$ satisfying $\u\times\bnu=0$ on 
$\Gamma_D$ in a generalized sense, 
and the space $W_H$ consists of all functions $\mathbf{v}$ such that 
$\mathbf{v}\times\bnu=0$ on $\Gamma_N$ in a weak sense. 
It is shown in Theorem~1, Chapter IX,
\S~3 of \cite{DaLi90} that the operator 
$$
  A:W_E\times W_H\to L^2(\Omega)^3\times L^2(\Omega)^3, \quad
	(\u ,\mathbf{v})\mapsto (-\curl\mathbf{v},\curl\u ),
$$
is skew self-adjoint, i.e.\ $A^*=-A$. Thus, $-\mathrm{i} A$ is self-adjoint and by the
Theorem of Stone, 
$-A$ generates a unitary $C_0$ group $(e^{-tA})_{t\in\R}$ in 
$L^2(\Omega)^3\times L^2(\Omega)^3$. 

The regularized Maxwell equations \eqref{fix.M1}-\eqref{fix.M2} can be reformulated
as
$$
  \pa_t(\E,\H) + A(\E,\H) = \sigma\big(D(\na R_\eps^x(\rho)-[\rho]_M\E),
	-\pa_t R_\eps^t(\m)\big), \quad t>0.
$$
The right-hand side is a function in $L^1(0,T;L^2(\Omega))^6$. Thus, 
by Duhamel's formula, 
$$
  (\E,\H)(t) = e^{-tA}(\E^0,\H^0) + \sigma\int_0^t e^{-(t-s)A}
	\big(D(\na R_\eps^x(\rho)-[\rho]_M\E),-\pa_t R_\eps^t(\m)\big)(s)ds.
$$
We infer that the solutions to \eqref{fix.M1}-\eqref{fix.M2} are the fixed points
of the operator $G:C^0([0,T];$ $L^2(\Omega))^6\to C^0([0,T];L^2(\Omega))^6$,
defined by
$$
  G(\E,\H) = e^{-tA}(\E^0,\H^0) + \sigma\int_0^t e^{-(t-s)A}
	\big(D(\na R_\eps^x(\rho)-[\rho]_M\E),-\pa_t R_\eps^t(\m)\big)(s)ds.
$$
Since $(e^{-tA})_{t\in\R}$ is a unitary group and $0\le[\rho]_M\le M$
in $\Omega\times(0,T)$, we obtain for $(\E,\H)$, 
$(\E',\H')\in C^0([0,T];L^2(\Omega))^6$,
\begin{align*}
  \big\| & G(\E,\H)-G(\E',\H')\big\|_{L^2(\Omega)^6}
	\le \bigg\|\int_0^t e^{-(t-s)A}(D[\rho]_M(\E-\E'),0)(s)ds\bigg\|_{L^2(\Omega)^6} \\
	&\le DM\int_0^t\|\E-\E'\|_{L^2(\Omega)^3}ds 
	\le DM T\|\E-\E'\|_{C^0([0,T];L^2(\Omega))^3}.
\end{align*}
Thus, choosing $T>0$ sufficiently small, $G$ becomes a contraction,
and there exists a unique local-in-time mild solution $(\E,\H)$ to 
\eqref{fix.M1}-\eqref{fix.M2}. The global solvability is a consequence of
the energy estimate (see e.g.\ \cite[Prop.~2.4]{JMR00}):
\begin{align*}
  \frac12\frac{d}{dt}\int_\Omega(|\E|^2+|\H|^2)dx
	&= \sigma\int_\Omega\big(D(\na R_\eps^x(\rho)-[\rho]_M\E)\cdot\E
	- \pa_t R_\eps^t(\m)\cdot\H\big)dx \\
	&\le \sigma\int_\Omega\big(D\na R_\eps^x(\rho)\cdot\E
	- \pa_t R_\eps^t(\m)\cdot\H\big)dx \\
	&\le \frac12\int_\Omega(|\E|^2+|\H|^2)dx + c\int_\Omega
	\big(|\na R_\eps^x(\rho)|^2 + |\pa_t R_\eps^t(\m)|^2\big)dx,
\end{align*}
where here and in the following, $c>0$ denotes a generic constant independent
of $\eps$ and $M$ if not stated otherwise. By Gronwall's lemma and the properties
\eqref{2.Rx1} and \eqref{2.Rt1} of the regularization
operators,
$$
  \int_\Omega(|\E(t)|^2+|\H(t)|^2)dx \le c + c(\eps)\int_0^t\int_\Omega
	(\rho^2+|\m|^2)dxds, \quad t\ge 0.
$$
This estimate allows us to continue the local solution
for all time $t>0$.
\end{proof}


\subsection{Analysis of the regularized LLG equation}\label{sec.llg}

We show that \eqref{fix.m} possesses a unique strong solution.

\begin{lemma}[Existence of the regularized LLG equation]\label{lem.llg}
Let $\H_\sigma:=\sigma(\H+\beta\s)\in L^2(\omega\times(0,T))$. 
Then there exists a unique strong solution
$\m$ to \eqref{fix.m}-\eqref{fix.bc.m} satisfying $|\m|\le 1$ in
$\omega\times(0,T)$,
$$
  \m\in L^\infty(0,T;H^1(\omega))\cap L^2(0,T;H^2(\omega)), \quad
	\pa_t\m\in L^2(0,T;L^2(\omega)),
$$
and the estimate
$$
  \|\m\|_{L^\infty(0,T;H^1(\omega))} + \eps^{1/2}\|\m\|_{L^2(0,T;H^2(\omega))}
	+ \|\pa_t\m\|_{L^2(0,T;L^2(\omega))} \le c,
$$
where $c>0$ is independent of $\eps$.
\end{lemma}

\begin{proof}
The proof is based on the Galerkin method, standard $L^2$ estimates,
and a Moser-type iteration to prove the $L^\infty$ bound for $\m$. 

{\em Step 1: Existence of solutions to \eqref{fix.m}-\eqref{fix.bc.m}.}
Let $e_i\in H^2(\omega)\cap L^\infty(\omega)$ ($i\in\N$)
be the eigenfunctions of $-\Delta$ in $\omega$ with homogeneous Neumann boundary 
conditions and with associated eigenvalues $\lambda_i>0$.  
Let $\m^{(N)}(x,t)=\sum_{j=1}^N \m^{(N,j)}(t)e_j(x)$ be the approximated solution
to \eqref{fix.m}-\eqref{fix.bc.m}, that is
\begin{align}\label{llg.1}
  \int_\omega\big( & \pa_t \m^{(N)} - \eps\Delta\m^{(N)} 
	- \m^{(N)}\times(\Delta\m^{(N)} + \H_\sigma) \\
	&{}+ \alpha\m^{(N)}\times(\m^{(N)}
	\times(\Delta\m^{(N)} + \H_\sigma))\big)e_i dx = 0, \quad i=1,\ldots,N. \nonumber
\end{align}
This is a system of ordinary differential equations
in the unknowns $(\m^{(N,j)})_{j=1,\ldots,N}$, which has a unique $H^2$ solution
$\m^{(N)}$ in a suitable time interval $(0,T^*)$ with $T^*\le T$. It remains
to find $N$-independent estimates for $\m^{(N)}$ in order to conclude
global solvability.

We take the inner product of \eqref{llg.1} and $\m^{(N,i)}$ and sum over $i$,
yielding
$$
  \frac12\frac{d}{dt}\int_\omega |\m^{(N)}|^2 dx 
	+ \eps\int_\omega|\na\m^{(N)}|^2 dx = 0
$$
and hence $\|\m^{(N)}(t)\|_{L^2(\omega)}\le \|\m^{(N)}(0)\|_{L^2(\omega)}\le c$
for $t<T^*$, where $c>0$ does not depend on $N$. Next, we take the inner product
of \eqref{llg.1} and $\lambda_i\m^{(N,i)}$, sum over $i$, and employ the
elementary itentity $(\mathbf{a}\times\mathbf{b})\cdot\mathbf{c}
=(\mathbf{c}\times\mathbf{a})\cdot\mathbf{b}$ for $\mathbf{a}$, $\mathbf{b}$,
$\mathbf{c}\in\R^3$:
\begin{align}
  \frac12\frac{d}{dt} & \int_\omega|\na\m^{(N)}|^2 dx
	+ \eps\int_\omega|\Delta\m^{(N)}|^2 dx 
	+ \alpha\int_\omega|\m^{(N)}\times\Delta\m^{(N)}|^2 dx \label{llg.2} \\
	&= \int_\omega \big(-\m^{(N)}\times\H_\sigma + \alpha\m^{(N)}\times(\m^{(N)}
	\times\H_\sigma)\big)\cdot\Delta\m^{(N)} dx \nonumber \\
	&= \int_\omega(-\H_\sigma + \alpha\m^{(N)}\times\H_\sigma)
	\cdot(\Delta\m^{(N)}\times\m^{(N)})dx \nonumber \\
	&\le \frac{\alpha}{2}\int_\omega|\m^{(N)}\times\Delta\m^{(N)}|^2 dx
	+ c\int_\omega|\H_\sigma|^2(1 + |\m^{(N)}|^2)dx \nonumber \\
	&\le \frac{\alpha}{2}\int_\omega|\m^{(N)}\times\Delta\m^{(N)}|^2 dx
	+ c(1 + \|\m^{(N)}\|_{L^\infty(\omega)}^2)
	\|\H_\sigma\|_{L^\infty(0,T;L^2(\omega))}^2,
	\nonumber
\end{align}
where $c>0$ is a generic constant independent of $N$.
By the Gagliardo-Nirenberg inequality, the $L^\infty$ norm of $\m^{(N)}$ 
can be bounded by
$$
  \|\m^{(N)}\|_{L^\infty(\omega)} \le c\|\m^{(N)}\|_{H^2(\omega)}^{1/2}
	\|\m^{(N)}\|_{L^2(\omega)}^{1/2}
	\le \delta\|\m^{(N)}\|_{H^2(\omega)} + c(\delta)\|\m^{(N)}\|_{L^2(\omega)},
$$
where $\delta>0$ is arbitrary.
The $H^2$ norm of $\m^{(N)}$ can be estimated by the $L^2$ norms of
$\Delta\m^{(N)}$ and $\m^{(N)}$ \cite[Lemma~2.1]{CaFa01}:
$$
  \|\m^{(N)}\|_{H^2(\omega)} \le c\big(\|\Delta\m^{(N)}\|_{L^2(\omega)}
	+ \|\m^{(N)}\|_{L^2(\omega)}\big)
	\le c\big(1+\|\Delta\m^{(N)}\|_{L^2(\omega)}\big),
$$
which holds for all $H^2$ functions with homogeneous Neumann boundary conditions.
Choosing $\delta>0$ sufficiently small, we infer from \eqref{llg.2} after
an integration over $(0,t)$ with $t<T^*$ that
$$
  \frac12\int_\omega|\na\m^{(N)}(t)|^2 dx 
	+ \frac{\eps}{2}\int_0^t\int_\omega|\Delta\m^{(N)}|^2 dxds
	+ \frac{\alpha}{2}\int_0^t\int_\omega|\m^{(N)}\times\Delta\m^{(N)}|^2 dxds
	\le c(\eps),
$$
where $c(\eps)>0$ depends on $\eps$ and $\|\na\m^{(N)}(0)\|_{L^2(\Omega)}$ 
but is independent of $N$. This shows that the solution $\m^{(N)}$ to \eqref{llg.1}
exists on $[0,T]$. Moreover, the above bound also allows us to perform the
limit $N\to\infty$ in \eqref{llg.1}, which gives a strong solution 
$\m\in L^\infty(0,T;H^1(\omega))\cap L^2(0,T;H^2(\omega))$, 
$\pa_t\m\in L^2(0,T;L^2(\omega))$ to \eqref{fix.m}-\eqref{fix.bc.m}.

{\em Step 2: $\eps$-uniform estimates for $\m$.}
Next, we prove some estimates for $\m$ which are independent of $\eps$.
First, we show the $L^\infty$ bound.
Let $p>1$. We take the inner product of \eqref{fix.m} and 
$|\m|^{p-1}\m\in L^2(0,T;L^\infty(\omega))$:
$$
  \frac{1}{p+1}\frac{d}{dt}\int_\omega|\m|^{p+1}dx
	+ \eps\int_\omega\sum_{k=1}^3\pa_k(|\m|^{p-1}\m)\cdot\pa_k\m dx = 0,
$$
where we abbreviated $\pa_k=\pa/\pa x_k$. Since
$$
  \sum_{k=1}^3\pa_k(|\m|^{p-1}\m)\cdot\pa_k\m
	= \sum_{i,j=1}^3\sum_{k=1}^3|\m|^{p-1}\bigg(\delta_{ij} 
	+ (p-1)\frac{m_im_j}{|\m|^2}\bigg)\pa_k m_i\pa_k m_j \ge 0,
$$
we obtain $\|\m(t)\|_{L^{p+1}(\omega)}\le \|\m^0\|_{L^{p+1}(\omega)}$ for all
$t\in[0,T]$ and $p>1$. Exploiting the fact that $|\m^0|=1$ in $\omega$, we
may let $p\to\infty$ to deduce that $\|\m(t)\|_{L^\infty(\omega)}\le 1$ for
$t\in[0,T]$.

In order to derive some uniform gradient estimates, we take the inner product
of \eqref{fix.m} and $-\Delta\m$ and integrate in $\omega$. Similarly as in
\eqref{llg.2}, this gives
\begin{align}
  \frac12\frac{d}{dt} & \int_\omega|\na\m|^2 dx + \eps\int_\omega|\Delta\m|^2 dx
	+ \frac{\alpha}{2}\int_\omega|\m\times\Delta\m|^2 dx \label{llg.aux1} \\
	&\le c\int_\omega|\H_\sigma|^2(1+|\m|^2) dx \le 2c\int_\omega|\H_\sigma|^2 dx, \nonumber
\end{align}
where we have used the fact that $|\m|\le 1$.
Taking the inner product of \eqref{fix.m} and $\pa_t\m$ and integrating in 
$\omega$ leads to
\begin{align*}
  \int_\omega|\pa_t\m|^2 dx + \frac{\eps}{2}\frac{d}{dt}\int_\omega|\na\m|^2 dx
	&= \int_\omega\pa_t\m\cdot(\m\times(\Delta\m+\H_\sigma))dx \\
	&\phantom{xx}{}
	- \alpha\int_\omega\pa_t\m\cdot\big(\m\times(\m\times(\Delta\m+\H_\sigma))\big)dx.
\end{align*}
We integrate over $(0,t)$, apply Young's inequality, use the boundedness of $\m$,
and take into account estimate \eqref{llg.aux1}:
\begin{align*}
  \int_0^t & \int_\omega|\pa_t\m|^2 dxds 
	+ \frac{\eps}{2}\int_\omega|\na\m(t)|^2 dx 
	- \frac{\eps}{2}\int_\omega|\na\m(0)|^2 dx  \\
	&\le \frac12\int_0^t \int_\omega|\pa_t\m|^2 dxds 
	+ c\int_0^t\int_\omega|\m\times\Delta\m|^2 dxds 
	+ c\int_0^t\int_\omega|\H_\sigma|^2 dxds \\
	&\le \frac12\int_0^t \int_\omega|\pa_t\m|^2 dxds 
	+ c\int_0^t\int_\omega|\H_\sigma|^2 dxds.
\end{align*}
Then, combining this estimate and the time-integrated version of
\eqref{llg.aux1}, we obtain
\begin{align}\label{llg.estm}
  \int_\omega & |\na\m(t)|^2 dx
	- \int_\omega|\na\m(0)|^2 dx
	+ \eps\int_0^t\int_\omega|\Delta\m|^2 dxds \\
	&{}+ \int_0^t\int_\omega|\pa_t\m|^2 dxds \le 
	c\int_0^t\int_\omega|\H_\sigma|^2 dxds. \nonumber
\end{align} 
This gives $\eps$-uniform estimates for $\m$ in the spaces
$L^\infty(0,T;H^1(\omega))$ and $H^1(0,T;L^2(\omega))$ and 
for $\eps^{1/2}\m$ in $L^2(0,T;H^2(\omega))$.

{\em Step 3: Uniqueness of solutions.} Let $\m$, $\m'$ be two strong 
solutions to \eqref{fix.m}-\eqref{fix.bc.m} satisfying $|\m|\le 1$,
$|\m'|\le 1$ in $\omega\times(0,T)$. Set $\u:=\m-\m'$ and recall that
$\H_\sigma=\sigma(\H+\beta\s)\in L^2(0,T;L^2(\omega))$. Then $\u$ solves
\begin{align*}
   \pa_t\u - \eps\Delta\u &= \m\times(\Delta\m+\H_\sigma)
	- \m'\times(\Delta\m'+\H_\sigma) \\
	&\phantom{xx}{}-\alpha\Big(\m\times(\m\times(\Delta\m+\H_\sigma))
	- \m'\times(\m'\times(\Delta\m'+\H_\sigma))\Big) \\
	&= \u\times(\Delta\m+\H_\sigma) + \m'\times\Delta\u 
	-\alpha\u\times(\m\times(\Delta\m+\H_\sigma))\\
	&\phantom{xx}- \alpha\m'\times(\u\times(\Delta\m+\H_\sigma)) 
	- \alpha\m'\times(\m'\times\Delta\u).
\end{align*}
Taking the inner product of this equation and $\u$, the first and third terms
on the right-hand side cancel. Then, integrating in $\omega$ leads to
\begin{align}
  \frac12\frac{d}{dt} & \int_\omega|\u|^2 dx + \eps\int_\omega|\na\u|^2 dx
	= \int_\omega\u\cdot(\m'\times\Delta\u)dx \label{llg.aux3} \\
	&{}- \alpha\int_\omega\u\cdot\big(\m'\times(\u\times(\Delta\m+\H_\sigma))\big)dx
	- \alpha\int_\omega\u\cdot(\m'\times(\m'\times\Delta\u))dx. \nonumber
\end{align}
Integrating by parts, 
the first two integrals on the right-hand side are estimated as follows:
\begin{align*}
  \int_\omega\u\cdot( & \m'\times\Delta\u)dx
	- \alpha\int_\omega\u\cdot\big(\m'\times(\u\times(\Delta\m+\H_\sigma))\big)dx \\
	&= -\int_\omega\u\cdot(\na\m'\times\na \u)dx 
	- \int_\omega\na\u\cdot(\m'\times\na\u)dx \\
	&\phantom{xx}{}+ \alpha\int_\omega\u\cdot(\m'\times(\na\u\times\na\m))dx
	+ \alpha\int_\omega\u\cdot(\na\m'\times(\u\times\na\m))dx \\
	&\phantom{xx}{}+ \alpha\int_\omega\na\u\cdot(\m'\times(\u\times\na\m))dx
	- \alpha\int_\omega\u\cdot(\m'\times(\u\times\H_\sigma))dx \\
	&\le c\int_\omega|\u|\big(|\na\m| + |\na\m'|\big)|\na\u| dx
	+ c\int_\omega|\u|^2|\na\m|\,|\na\m'|dx + c\int_\omega|\u|^2|\H_\sigma|dx.
\end{align*}
The last integral on the right-hand side of \eqref{llg.aux3} can be treated
in a similar way:
\begin{align*}
  - \alpha\int_\omega & \u\cdot(\m'\times(\m'\times\Delta\u))dx
	= \alpha\int_\omega\u\cdot(\m'\times(\na\m'\times\na\u))dx \\
	&{}+ \alpha\int_\omega\u\cdot(\na\m'\times(\m'\times\na\u))dx
	+ \alpha \int_\omega\na\u\cdot(\m'\times(\m'\times\na\u))dx \\
	&\le c\int_\omega|\u|\,|\na\m'|\,|\na\u| dx - \alpha\int_\omega|\m\times\na\u|^2 dx.
\end{align*}
Inserting these estimates into \eqref{llg.aux3}, we deduce that
\begin{align}
  \frac12\frac{d}{dt}\int_\omega|\u|^2 dx + \eps\int_\omega|\na\u|^2 dx
	&\le c\int_\omega|\u|\big(|\na\m| + |\na\m'|\big)|\na\u| dx \label{llg.aux4} \\
	&\phantom{xx}{}
	+ c\int_\omega|\u|^2|\na\m|\,|\na\m'|dx + c\int_\omega|\u|^2|\H_\sigma|dx.
	\nonumber
\end{align}

The first integral on the right-hand side becomes
$$
  \int_\omega|\u|\big(|\na\m| + |\na\m'|\big)|\na\u| dx
	\le \|\u\|_{L^4(\omega)}\big(\|\na\m\|_{L^4(\omega)} + \|\na\m'\|_{L^4(\omega)}\big)
	\|\na\u\|_{L^2(\omega)}.
$$
The Gagliardo-Nirenberg inequality and the bound $|\m|\le 1$ imply that
\begin{align*}
  \|\u\|_{L^4(\omega)} &\le c\|\u\|_{H^1(\omega)}^{1/2}\|\u\|_{L^2(\omega)}^{1/2}, \\
	\|\na\m\|_{L^4(\omega)} &\le c\|\m\|_{H^2(\omega)}^{1/2}
	\|\m\|_{L^\infty(\omega)}^{1/2} \le c\|\m\|_{H^2(\omega)}^{1/2}, \\
	\|\na\m'\|_{L^4(\omega)} &\le c\|\m'\|_{H^2(\omega)}^{1/2}
	\|\m'\|_{L^\infty(\omega)}^{1/2} \le c\|\m'\|_{H^2(\omega)}^{1/2}.
\end{align*}
This shows that
\begin{align*}
  \int_\omega|\u|\big(|\na\m| + |\na\m'|\big)|\na\u| dx
	&\le c\|\u\|_{L^2(\omega)}^{1/2}\big(\|\m\|_{H^2(\omega)}^{1/2}
	+ \|\m'\|_{H^2(\omega)}^{1/2}\big)\|\u\|_{H^1(\omega)}^{3/2} \\
	&\le \delta\|\u\|_{H^1(\omega)}^2 + c(\delta)\|\u\|_{L^2(\omega)}^2
	\big(\|\m\|_{H^2(\omega)}^{2}	+ \|\m'\|_{H^2(\omega)}^{2}\big),
\end{align*}
where $\delta>0$ is arbitrary. This argument is only possible in two space
dimensions. Indeed, in three dimensions, we obtain the expression
$c(\delta)\|\u\|_{L^2(\omega)}^2\big(\|\m\|_{H^2(\omega)}^{4}	
+ \|\m'\|_{H^2(\omega)}^{4}\big)$, which we cannot estimate since
we do not have the regularity $\m$, $\m'\in L^4(0,T;H^2(\omega))$.
 
The second integral on the right-hand side of \eqref{llg.aux4} can be
estimated in a similar way, using the continuous embedding 
$H^1(\omega)\hookrightarrow L^{4}(\omega)$:
\begin{align*}
  \int_\omega|\u|^2|\na\m|\,|\na\m'|dx
	&\le \|\u\|_{L^4(\omega)}^2\|\na\m\|_{L^4(\omega)}\|\na\m'\|_{L^4(\omega)} \\
	&\le \delta\|\u\|_{H^1(\omega)}^2 
	+ c(\delta)\|\u\|_{L^2(\omega)}^2\|\m\|_{H^2(\omega)}\|\m'\|_{H^2(\omega)}.
\end{align*}
(Also this estimate holds in two space dimensions only.)
Finally, the last integral in \eqref{llg.aux4} becomes
\begin{align*}
  \int_\omega|\u|^2|\H_\sigma|dx 
	&\le \|\u\|_{L^4(\omega)}^2\|\H_\sigma\|_{L^2(\omega)}
	\le c\|\u\|_{H^1(\omega)}\|\u\|_{L^2(\omega)}\|\H_\sigma\|_{L^2(\omega)} \\
	&\le \delta\|\u\|_{H^1(\omega)}^2 + c(\delta)\|\H_\sigma\|_{L^2(\omega)}^2
	\|\u\|_{L^2(\omega)}^2.
\end{align*}
Putting these estimates together and choosing $\delta>0$
sufficiently small, we conclude from \eqref{llg.aux4} that
\begin{align*}
  & \frac12\frac{d}{dt}\int_\omega|\u|^2 dx + \frac{\eps}{2}\int_\omega|\na\u|^2 dx
	\le c(\eps)\big(1+g(t)\big)\int_\omega|\u|^2 dx,  \\
	& \mbox{where }g(t) = \|\H_\sigma(t)\|_{L^2(\omega)}^2 + \|\m(t)\|_{H^2(\omega)}^2
	+ \|\m'(t)\|_{H^2(\omega)}^2.
\end{align*}
As $g\in L^1(0,T)$, 
Gronwall's lemma and $\u(0)=0$ imply that $\u(t)=0$ in $\omega$, $t>0$, which
finishes the proof. 
\end{proof}


\subsection{Uniform estimates and existence of the regularized problem}
\label{sec.est}

We need uniform estimates for all fixed points of the operator $F$, defined in 
Section \ref{sec.strat}. Such an estimate is provided by the following lemma.
Recall that $X=L^2(0,T;L^2(\Omega))^3$ and $Y=C^0([0,T];L^2(\omega))$.

\begin{lemma}[$L^2$ estimate]\label{lem.est}
Let $(\rho,\s;\m)\in X\times Y$ 
be a fixed point of $F(\cdot,\sigma)$ for some $\sigma\in[0,1]$.
Then there exist constants $c_1$, $c_2(M)>0$, which are independent of
$\eps$ and $\sigma$, such that for all $t\in(0,T)$, 
the functional $S(t)$, defined in \eqref{1.S}, satisfies
$$
  S(t) + c_1\int_0^t\int_\Omega\big(|\na\rho|^2 + |\na\s|^2\big)dx
	+ c_1\int_0^t\int_\omega|\pa_t\m|^2 dx \le c_2(M).
$$
\end{lemma}

\begin{proof}
To simplify the computations, we let $\sigma=1$. The proof for general
$\sigma\in[0,1]$ is similar. We compute
\begin{equation}\label{en.1}
  \frac{dS}{dt} = \langle\pa_t\rho,\rho-\rho_D\rangle + \langle\pa_t\s,\s\rangle
	+ \frac12\frac{d}{dt}\int_\Omega\big(|\E|^2 + |\H|^2\big) dx
	+ \frac12\frac{d}{dt}\int_\omega|\na\m|^2 dx,
\end{equation}
where $\langle\cdot,\cdot\rangle$ is the dual product between
$H_D^1(\Omega)'$ and $H_D^1(\Omega)$.
Employing \eqref{fix.rho} with $\sigma=1$ and $\rho=\rho^*$, we find that
\begin{align}
  \int_0^t & \langle\pa_t\rho,\rho-\rho_D\rangle ds
	= -\int_0^t\int_\Omega D\na(\rho-\rho_D)\cdot(\na\rho-[\rho]_M\E) dxds 
	\label{en.rho}\\
	&\le -\frac12\int_0^t\int_\Omega D|\na\rho|^2 dxds
	+ \frac12\int_0^t\int_\Omega D|\na\rho_D|^2 dxds \nonumber \\
	&\phantom{xx}{}+ cM\int_0^t\int_\Omega|\na(\rho-\rho_D)|\,|\E|dxds \nonumber \\
	&\le -c\int_0^t\int_\Omega|\na\rho|^2 dxds + c\int_0^t\int_\Omega|\E|^2 dxds + c(M).
	\nonumber
\end{align}
Furthermore, using \eqref{fix.s} with $\sigma=1$, $\s^*=\s$, and $\m=\m^*$,
\begin{align}
  \int_0^t \langle\pa_t\s,\s\rangle ds
	&= -\int_0^t\int_\Omega D|\na\s|^2 dxds 
	+ \int_0^t\int_\Omega D[|\s|]_M\na\s:\bigg(\frac{\s}{|\s|}\otimes\E\bigg)dxds 
	\label{en.s} \\
	&\phantom{xx}{}- \int_0^t\int_\Omega\frac{|\s|^2}{\tau}dxds \nonumber \\
	&\le -\int_0^t\int_\Omega D|\na\s|^2 dxds 
	+ c(M)\int_0^t\int_\Omega|\na\s|\,|\E|dxds
	- \int_0^t\int_\Omega\frac{|\s|^2}{\tau}dxds \nonumber \\
	&\le -c\int_0^t\int_\Omega |\na\s|^2 dxds + c(M)\int_0^t\int_\Omega|\E|^2 dxds.
	\nonumber
\end{align}
Next, we have the energy estimate for the Maxwell equations 
\eqref{fix.M1}-\eqref{fix.M2}:
\begin{align*}
  \frac12\int_0^t & \frac{d}{dt}\int_\Omega\big(|\E|^2 + |\H|^2\big)dxds \\
	&= \int_0^t\int_\omega D\big(\na R_\eps^x(\rho)\cdot\E - [\rho]_M|\E|^2\big)dxds
	- \int_0^t\int_\omega\pa_t R_\eps^t(\m)\cdot\H dxds \\
	&\le \delta\int_0^t\int_\omega\big(|\na R_\eps^x(\rho)|^2 + |\pa_t R_\eps^t(\m)|^2
	\big)dxds + c(\delta)\int_0^t\int_\omega |\H|^2 dxds,
\end{align*}
where $\delta>0$ is arbitrary. By the properties \eqref{2.Rx2} and \eqref{2.Rt2}
of the regularization operators, it follows that
\begin{align*}
  \frac12\int_0^t \frac{d}{dt}\int_\Omega\big(|\E|^2 + |\H|^2\big)dx
  &\le \delta\int_0^t\|\rho(s)\|_{H^1(\Omega)}^2 ds 
	+ \delta\int_0^t\|\pa_t\m(s)\|_{L^2(\omega)}^2 ds \\
	&\phantom{xx}{}+ \delta\int_0^t\|\m(s)\|_{L^2(\omega)}^2 ds
	+ c(\delta)\int_0^t\int_\omega |\H|^2 dxds.
\end{align*}
The $L^2$ norm of $\na\rho$ can be absorbed by
the corresponding term in \eqref{en.rho},
choosing $\delta>0$ sufficiently small. 
Estimate \eqref{llg.estm} can be formulated as
$$
  \frac{d}{dt}\int_\omega|\na\m(t)|^2 dx + \eps\int_\omega|\Delta\m|^2 dx
	+ \int_\omega|\pa_t\m|^2 dx \le c\int_\omega(|\H|^2+|\s|^2) dx.
$$
Combining the above estimates, \eqref{en.1} becomes, after time integration,
$$
  S(t) + c\int_0^t\int_\Omega\big(|\na\rho|^2 + |\na\s|^2\big)dxds 
	+ \int_0^t\int_\omega|\pa_t\m|^2 dxds \le c + c(M)\int_0^t S(s)ds.
$$
An application of Gronwall's lemma ends the proof.
\end{proof}

\begin{corollary}[Solution of the regularized problem]\label{coro.reg}
There exists a weak solution $(\rho,\s,\E,$ $\H,\m)$ to \eqref{2.rho}-\eqref{2.m}
satisfying the initial and boundary conditions \eqref{1.ibc.m}, 
\eqref{1.ic.M}-\eqref{1.bc.M1}, and \eqref{1.ic.rho}-\eqref{1.bc.rhoN}
as well as the regularity properties stated in Theorem \ref{thm.ex}.
\end{corollary}

\begin{proof}
Lemma \ref{lem.est} provides uniform estimates for the fixed-point operator
$F$ defined in Section \ref{sec.strat}. Thus, the result follows from the
Leray-Schauder fixed-point theorem.
\end{proof}


\subsection{The limit $\eps\to 0$}\label{sec.eps}

The estimate in Lemma \ref{lem.est} is independent of $\eps$, which
allows us to perform the limit $\eps\to 0$.

\begin{lemma}\label{lem.eps}
There exists a weak solution $(\rho,\s,\E,\H,\m)$ to 
\begin{align}
  & \pa_t\rho - \diver(D(\na\rho-[\rho]_M\E) = 0, \label{eps.rho} \\
  & \pa_t\s - \diver\bigg(D\bigg(\na\s-[|\s|]_M\frac{\s}{|\s|}\otimes\E\bigg)\bigg)
	+ \gamma\m\times[|\s|]_M\frac{\s}{|\s|} = -\frac{\s}{\tau}, \label{eps.s} \\
	& \pa_t\E - \curl\H = D(\na\rho-[\rho]_M\E), \label{eps.E} \\
	& \pa_t\H + \curl E = -\pa_t\m \quad\mbox{in }\Omega\times(0,T), \label{eps.H} \\
	& \pa_t\m = \m\times(\Delta\m+\H+\beta\s) 
	- \alpha\m\times(\m\times(\Delta\m+\H+\beta\s))
	\quad\mbox{in }\omega\times(0,\omega), \label{eps.m}
\end{align}
with the initial and boundary conditions \eqref{1.ibc.m}, 
\eqref{1.ic.M}-\eqref{1.bc.M1}, \eqref{1.ic.rho}-\eqref{1.bc.rhoN},
satisfying the regularity properties stated in Theorem \ref{thm.ex} 
and the constraint $|\m|=1$ in $\omega\times(0,T)$.
\end{lemma}

\begin{proof}
We denote the solution to \eqref{2.rho}-\eqref{2.m} with a superindex $\eps$
to indicate the dependence on this parameter. Lemma \ref{lem.est}
gives the uniform estimates
\begin{align*}
  \|\rho^{(\eps)}\|_{L^2(0,T;H^1(\Omega))} + \|\s^{(\eps)}\|_{L^2(0,T;H^1(\Omega))}
  &\le c, \\
	\|\m^{(\eps)}\|_{L^\infty(0,T;H^1(\omega))} 
	+ \|\pa_t\m^{(\eps)}\|_{L^2(0,T;L^2(\omega))}
	&\le c, \\
	\|\E^{(\eps)}\|_{L^\infty(0,T;L^2(\Omega))} 
	+ \|\H^{(\eps)}\|_{L^\infty(0,T;L^2(\Omega))} &\le c.
\end{align*}
These estimates and \eqref{2.rho}-\eqref{2.s} show that
$$
  \|\pa_t\rho^{(\eps)}\|_{L^2(0,T;H_D^1(\Omega)')} 
	+ \|\pa_t\s^{(\eps)}\|_{L^2(0,T;H^1(\Omega)')}
	\le c,
$$
where the constant $c>0$ may depend on the truncation parameter $M$ but not on $\eps$.
Therefore, we infer from the Aubin lemma \cite{Sim87} and weak compactness
that, up to subsequences which are not relabeled, as $\eps\to 0$,
\begin{align*}
  \rho^{(\eps)}\to \rho, \ \s^{(\eps)}\to\s &\quad\mbox{strongly in }
	L^2(0,T;L^2(\Omega)), \\
	\rho^{(\eps)}\rightharpoonup\rho,\ \s^{(\eps)}\rightharpoonup\s
	&\quad\mbox{weakly in }L^2(0,T;H^1(\Omega)), \\
	\pa_t\rho^{(\eps)}\rightharpoonup\pa_t\rho &\quad\mbox{weakly in }
	L^2(0,T;H_D^1(\Omega)'), \\
	\pa_t\s^{(\eps)}\rightharpoonup\pa_t\s &\quad\mbox{weakly in }
	L^2(0,T;H^1_D(\Omega)'), \\
	\E^{(\eps)}\rightharpoonup^*\E, \ \H^{(\eps)}\rightharpoonup^*\H
	&\quad\mbox{weakly* in }L^\infty(0,T;L^2(\Omega)), \\
  \m^{(\eps)}\to\m &\quad\mbox{strongly in }C^0([0,T];L^p(\omega)), \ p<\infty, \\
	\m^{(\eps)}\rightharpoonup^*\m &\quad\mbox{weakly* in }L^\infty(0,T;H^1(\omega)), \\
	\pa_t\m^{(\eps)}\rightharpoonup\pa_t\m &\quad\mbox{weakly in }
	L^2(0,T;L^2(\omega)).
\end{align*}
According to \cite[p.~671]{Joc96}, $\na R_\eps^x(\rho^{(\eps)})$ and
$\pa_tR_\eps^t(\m^{(\eps)})$ converge weakly in $L^2$ to $\na\rho$ and
$\pa_t\m$, respectively, taking into account \eqref{2.Rx2}, \eqref{2.Rt2}.
These convergences allow us to perform the limit $\eps\to 0$ in
\eqref{2.rho}-\eqref{2.H}, showing that $(\rho,\s,\E,\H)$ is a weak solution
to \eqref{eps.rho}-\eqref{eps.H}.

It remains to pass to the limit $\eps\to 0$ in the regularized
LLG equation \eqref{2.m}. For this, we observe
that \eqref{2.m} can be rewritten as 
$\mathbf{v}-\alpha \m^{(\eps)}\times\mathbf{v}=\mathbf{f}_\eps$,
where $\mathbf{v}=\m^{(\eps)}\times(\Delta\m^{(\eps)}+\H^{(\eps)}+\beta\s^{(\eps)})$ 
and $\mathbf{f}_\eps=\pa_t\m^{(\eps)}-\eps\Delta\m^{(\eps)}$.
The solution of this equation is $\mathbf{v}=G(\alpha\m^{(\eps)})\mathbf{f}_\eps$,
where 
$$
  G(\alpha\m^{(\eps)}):\R^3\to\R^3, \quad
	G(\alpha\m^{(\eps)})\mathbf{f} = (1+|\mathbf{v}|^2)^{-1}\big(\mathbf{f} 
	+ \alpha\m^{(\eps)}\times\mathbf{f} 
	+ (\alpha\m^{(\eps)}\cdot\mathbf{f})\mathbf{f}\big),
$$
is the inverse of the mapping $\u\mapsto \u-\alpha\m^{(\eps)}\times\u$ for
$\u\in\R^3$. Thus, \eqref{2.m} rewrites as
$$
  G(\alpha\m^{(\eps)})\big(\pa_t\m^{(\eps)}-\eps\Delta\m^{(\eps)}\big)
	= \m^{(\eps)}\times(\Delta\m^{(\eps)}+\H^{(\eps)}+\beta\s^{(\eps)}).
$$
Multiplying this equation with the test function $\phi\in C^\infty(\omega)$ 
and integrating over $\omega$, integration by parts leads to
\begin{align*}
  \int_\omega & G(\alpha\m^{(\eps)})\pa_t\m^{(\eps)} \phi dx
	+ \eps\int_\omega\big(G(\alpha\m^{(\eps)})\na\m^{(\eps)}\cdot\na\phi
	+ \alpha\phi G'(\alpha\m^{(\eps)})|\na\m^{(\eps)}|^2\big)dx \\
	&= -\int_\omega\na\phi\cdot(\m^{(\eps)}\times\na\m^{(\eps)}) dx
	+ \int_\omega\phi\m^{(\eps)}\times(\H^{(\eps)}+\beta\s^{(\eps)})dx.
\end{align*}
By the above convergence results, we can let $\eps\to 0$ to obtain
$$
  \int_\omega G(\alpha\m)\pa_t\m \phi dx = -\int_\omega\na\phi\cdot(\m\times\na\m) dx
	+ \int_\omega\phi\m\times(\H+\beta\s)dx
$$
a.e.\ in $(0,T)$, which is the weak formulation of \eqref{eps.m}. 
Taking the inner product of \eqref{eps.m}
and $\m$, we conclude immediately that $|\m(t)|=|\m(0)|=1$ in $\omega\times(0,T)$.
\end{proof}

We conclude this section by showing that \eqref{1.M2} and \eqref{1.bc.M2} hold
in a weak sense. Let $\phi\in C^\infty_0(\Omega\backslash\overline{\Gamma}_D)$.
Taking the inner product of \eqref{eps.E} and $\na\phi$, integrating over
$\Omega$, and employing \eqref{eps.rho} gives
$$
  \frac{d}{dt}\int_\Omega\E\cdot\na\phi dx
	= \int_\Omega\curl\H\cdot\na\phi dx
	+ \int_\Omega D(\na\rho-[\rho]_M\E)\cdot\na\phi dx
  = -\frac{d}{dt}\int_\Omega\rho\phi dx.
$$
Consequently,
\begin{equation}\label{eps.aux}
  \frac{d}{dt}\int_\Omega(\diver\E-\rho)\phi dx 
	= \frac{d}{dt}\int_{\Gamma_N}\E\cdot\bnu\phi d\sigma.
\end{equation}
If $\phi\in C_0^\infty(\Omega)$,
we infer that $\diver\E(t)-\rho(t)$ is constant in time. Taking into account
the first equation in \eqref{hypo.ic}, it follows that $\diver\E-\rho(t)=-C(x)$ holds
for all $t\ge 0$. When $\phi\in C^\infty_0(\Omega\backslash\overline{\Gamma}_D)$,
\eqref{eps.aux} becomes $(d/dt)\int_{\Gamma_N}\E\cdot\bnu\phi d\sigma=0$
and then the third equation in \eqref{hypo.ic} shows that $\E(t)\cdot\bnu=0$
on $\Gamma_N$ for $t\ge 0$.
Finally, taking the divergence of \eqref{eps.H}, it follows
that $ \pa_t\diver(\H+\m)=0$ in the sense of distributions
and, because of the second equation in \eqref{hypo.ic},
$\diver(\H(t)+\m(t))=0$ for $t\ge 0$. 


\subsection{Uniform $L^\infty$ bounds for the charge and spin densities}
\label{sec.infty}

We show that $\rho$ and $\s$ are bounded uniformly in $M$ which allows us to remove
the truncation in \eqref{eps.rho}-\eqref{eps.E}.

\begin{lemma}\label{lem.infty}
Let $(\rho,\s,\E,\H,\m)$ be a weak solution to \eqref{eps.rho}-\eqref{eps.m}
satisfying the initial and boundary conditions \eqref{1.ibc.m}, 
\eqref{1.ic.M}-\eqref{1.bc.M1}, \eqref{1.ic.rho}-\eqref{1.bc.rhoN}, and equations
\eqref{1.M2}, \eqref{1.bc.M2}, $|\m|=1$ in $\omega\times(0,T)$.
Then $\rho\ge 0$ in $\Omega\times(0,T)$ and there exists $c(T)>0$,
independent of the truncation parameter $M$ and the parameter $\beta$, such that
$$
  \|\rho\|_{L^\infty(0,T;L^\infty(\Omega))} + \|\s\|_{L^\infty(0,T;L^\infty(\Omega))}
	\le c(T).
$$
\end{lemma}

\begin{proof}
Using $\min\{0,\rho\}$ as a test function in \eqref{eps.rho}, it follows immediately
that $\rho\ge 0$. For the proof of $L^\infty$ bounds for $\rho$ and $\s$, 
we employ a Moser-type iteration method. For this, let $p\ge 2$. We employ
the test function $[|\s|]_M^{p-1}\s/|\s|\in L^2(0,T;H_0^1(\Omega))$ in the weak
formulation of \eqref{eps.s}. Observing that
$$
  \phi_{p,M}(\sigma) = \int_0^\sigma [u]_Mdu \ge \frac{1}{p}[\sigma]_M^p
	\quad\mbox{for }\sigma\ge 0,
$$
we find that
\begin{align}
  \frac{d}{dt} \phi_{p,M}(|\s|)dx 
	&= \bigg\langle\pa_t\s,[|\s|]_M^{p-1}\frac{\s}{|\s|}\bigg\rangle \nonumber \\
	&= -\sum_{i=1}^3 \int_\Omega D\bigg(\na s_i-[|\s|]_M\frac{\s}{|\s|}\E\bigg)
	\cdot\na\bigg([|\s|]_M^{p-1}\frac{\s}{|\s|}\bigg)dx 
	- \int_\Omega[|\s|]_M^{p-1}\frac{|\s|^2}{\tau}dx \label{inf.ddt} \\
	&= I_1 + I_2. \nonumber
\end{align}
The integral $I_2$ is clearly nonpositive. Since
\begin{equation}\label{nas}
  \na\frac{s_j}{|\s|} = \sum_{k=1}^3\bigg(\delta_{jk} - \frac{s_js_k}{|\s|^2}\bigg)
	\frac{\na s_k}{|\s|},
\end{equation}
we can write the remaining integral $I_1$ as
\begin{align*}
  I_1 &= -\sum_{i=1}^3 \int_\Omega D\bigg(\na s_i-[|\s|]_M\frac{s_i}{|\s|}\E\bigg) \\
	&\phantom{xx}{}\times\bigg(\na\big([|\s|]_M^{p-1}\big)\frac{s_i}{|\s|}
	+ [|\s|]_M^{p-1}\sum_{j=1}^3\frac{1}{|\s|}\bigg(\delta_{ij}-\frac{s_is_j}{|\s|^2}
	\bigg)\na s_j\bigg)dx \\
	&= -\int_\Omega D\na|\s|\cdot\na\big([|\s|]_M^{p-1}\big)dx
	- \sum_{i,j=1}^3\int_\Omega D\frac{[|\s|]_M^{p-1}}{|\s|}
	\bigg(\delta_{ij}-\frac{s_is_j}{|\s|^2}\bigg)\na s_i\cdot\na s_j dx \\
	&\phantom{xx}{}+ \int_\Omega D[|\s|]\E\cdot\na\big([|\s|]_M^{p-1}\big)dx 
	=: I_{3} + I_{4} + I_{5}.
\end{align*}

We show that $I_3\le 0$. Indeed, the definition of $[{}\cdot{}]_M$ implies that
$\na[|\s|]_M = \chi_{\{|\s|\le M\}}\na|\s|$, which yields
$$
  I_3 = -\int_\Omega D(p-1)[|\s|]_M^{p-2}\chi_{\{|\s|\le M\}}\big|\na|\s|\big|^2 dx
	\le 0.
$$
Furthermore, since the matrix $\mathbb{I}_{3\times 3}-\s\otimes\s/|\s|^2$
is positive semidefinite, $I_4\le 0$. For the final estimate of $I_5$,
we need the assumption that $D$ is constant (also see Remark \ref{rem.Linfty}).
Then, integrating by parts and employing the first equation in \eqref{1.M2}, 
we obtain
\begin{align*}
  I_5 &= D(p-1)\int_\Omega [|\s|]_M^{p-1}\E\cdot\na[|\s|]_M dx \\
	&= \frac{p-1}{p}D\int_\Omega\E\cdot\na\big([|\s|]_M^p\big)dx
	= -\frac{p-1}{p}D\int_\Omega[|\s|]_M^p (\rho-C)dx \\
	&\ge \frac{p-1}{p}D\|C\|_{L^\infty(\Omega)}\int_\Omega[|\s|]_M^p dx
	\ge (p-1)D\|C\|_{L^\infty(\Omega)}\int_\Omega\phi_{p,M}(|\s|)dx.
\end{align*}

Therefore, \eqref{inf.ddt} becomes
$$
  \frac{d}{dt}\int_\Omega\phi_{p,M}(|\s|)dx 
	\le (p-1)D\|C\|_{L^\infty(\Omega)}\int_\Omega\phi_{p,M}(|\s|)dx,
$$
and Gronwall's lemma allows us to conclude that
\begin{align*}
  \int_\Omega[|\s|]_M^p dx 
	&\le p\int_\Omega\phi_{p,M}(|\s(\cdot,t)|)dx \\
	&\le p\exp\big((p-1)D\|C\|_{L^\infty(\Omega)}t\big)
	\int_\Omega\phi_{p,M}(|\s(\cdot,0)|)dx \\
	&= p\exp\big((p-1)D\|C\|_{L^\infty(\Omega)}t\big)\int_\Omega|\s^0|^p dx,
\end{align*}
since $|\s^0|\le M$. Taking the $p$th root and passing to the limit $p\to \infty$,
we infer that
$$
  \big\|[|\s(\cdot,t)|]_M\big\|_{L^\infty(\Omega)}
	\le \exp(D\|C\|_{L^\infty(\Omega)}t)\|\s^0\|_{L^\infty(\Omega)}, \quad t\ge 0.
$$

Now, we choose $M>M_T:=\exp(D\|C\|_{L^\infty(\Omega)}T)\|\s^0\|_{L^\infty(\Omega)}$
and define $\Omega_M(t)=\{x\in \Omega:|\s(x,t)|>M\}$. If $\Omega_M(t)$ has positive
Lebesgue measure for some $0\le t\le T$, then
$$
  M < \big\|[|\s(\cdot,t)|]_M\big\|_{L^\infty(\Omega)}
	\le \exp(D\|C\|_{L^\infty(\Omega)}t)\|\s^0\|_{L^\infty(\Omega)} \le M_T < M,
$$
which is absurd. Thus, $\Omega_M(t)$ is a set of measure zero for a.e.\ $t\in(0,T)$,
which implies that $|\s(x,t)|\le M$ for a.e.\ $x\in\Omega$, $t\in(0,T)$.
Since $M$ is arbitrary in the interval $(M_T,\infty)$, we conclude that
$$
  \|\s(\cdot,t)\|_{L^\infty(\Omega)}\le M_T, \quad t\in(0,T).
$$

The proof of the boundedness of $\rho$ is similar using $(([\rho]_M-K)^+)^{p-1}$
with $M\ge K:=\max\{\|\rho_D\|_{L^\infty(\Gamma_D)},\|\rho^0\|_{L^\infty(\Omega)}\}$ 
as a test function in \eqref{fix.rho} (see \cite{Jue97}).
\end{proof}

\begin{remark}[Generalizations]\label{rem.Linfty}\rm
The boundedness result can be generalized using refined Moser iteration techniques.
For instance, following the proof of \cite{Joc96}, 
we may allow for nonconstant diffusion 
coefficients $D(x)$ in case that the electric field $\E$
is given. It turns out that the $L^\infty$ bounds of $\rho$ and $\s$ depend on the
$L^\infty(0,T;L^2(\Omega))$ norm of $\E$. Since in our proof, this norm depends
on the truncation parameter $M$, we cannot conclude the proof 
but the argument is valid
if the Maxwell equations are replaced by {\em given} functions $\E$ and $\H$.

It is possible to prove that the $L^\infty$ bounds for $\rho$ and $\s$  
are also uniform in time.
The idea is to exploit the gradient norm in $I_3$. Using $|\s|^{p-2}\s$
as a test function in the weak formulation of \eqref{eps.s} (this is possible
since we already know that $\s$ is bounded locally in time), we find after
some elementary computations that
\begin{align*}
  \frac{1}{p}\,\frac{d}{dt}\int_\Omega|\s|^{p}dx
	&+ 4\frac{p-1}{p^2}D\int_\Omega \big|\na|\s|^{p/2}\big|^2 dx \\
	&= \frac{p-1}{p}D\int_\Omega\E\cdot\na|\s|^{p}dx 
	- \frac{1}{\tau}\int_\Omega|\s|^{p} dx.
\end{align*}
Neglecting the last integral, integrating by parts in the first integral
on the right-hand side, and employing \eqref{1.M2},
$$
  \frac{d}{dt}\int_\Omega|\s|^{p}dx
	+ 4\frac{p-1}{p}D\int_\Omega \big|\na|\s|^{p/2}\big|^2 dx 
	\le (p-1)D\|C\|_{L^\infty(\Omega)}\int_\Omega|\s|^p dx.
$$
By the Gagliardo-Nirenberg inequality, we may replace the $L^p$ norm of
$\s$ on the right-hand side by its $L^{p/2}$ norm (by absorbing the
$L^2$ gradient norm of $|\s|^{p/2}$ by the corresponding term on the left-hand side).
This yields a sequence of recursive inequalities of the type
$$
  \frac{dz_p}{dt} \le c_1 pz_{p/2}^2 + c_2, 
  \quad\mbox{where }z_p=\|\s\|_{L^p(\Omega)}^p.
$$
The strategy of the rest of the proof is to derive iteratively bounds
for $z_{2^m}$ for all $m\in\N$, which are uniform in $m$, and to pass to 
the limit $m\to\infty$. This can be done exactly as in \cite{HPS07}.
This idea goes back to Alikakos \cite{Ali79}. The result is the estimate
$$
  \|\s(\cdot,t)\|_{L^\infty(\Omega)} \le c\max\{1,\|\s^0\|_{L^\infty(\Omega)}\},
	\quad t\ge 0,
$$
where the constant $c>0$ only depends on $\|C\|_{L^\infty(\Omega)}$.
\qed
\end{remark}


\section{Free energy estimate}\label{sec.en}

We show that the relative free energy \eqref{1.E} is nonincreasing in time
under certain conditions. First, we comment on the spin contribution of the
energy. It comes from the von-Neumann entropy density
$\mbox{tr}(N\log N-N)$, where ``tr'' is the trace of a matrix and
$N=\rho\sigma_0 + \s\cdot\boldsymbol{\sigma}$ is the density matrix,
which is a Hermitian $2\times 2$ matrix.
Here, $\sigma_0$ denotes the identity matrix and $\boldsymbol{\sigma}
=(\sigma_1,\sigma_2,\sigma_3)$ is the vector of the Pauli matrices (see
\cite[Formula (1)]{PoNe11} for a definition). We may decompose $N$ according to
$N=\rho_+\Pi_+ + \rho_-\Pi_-$, where $\rho_\pm = \rho\pm|\s|$ are the
eigenvalues of $N$ and $\Pi_\pm=\frac12(\sigma_0\pm(\s/|\s|)\cdot\boldsymbol{\sigma})$
are the projections on the corresponding eigenspaces, 
satisfying $\Pi_\pm^2=\Pi_\pm$ and $\Pi_+\Pi_-=0$. Then, by spectral theory,
$$
  N\log N-N = \rho_+(\log\rho_+-1) + \rho_-(\log\rho_--1),
$$
which is the expression used in \eqref{1.E}.

\begin{proposition}[Monotonicity of the free energy]\label{prop.en}
Let $(\rho,\s,\E,\H,\m)$ be a smooth solution to \eqref{1.m}-\eqref{1.bc.rhoN}
satisfying $\rho>|\s|$. Furthermore, let $\|\rho(t)\|_{L^\infty(\Omega)}\le M(T)$,
where $M(T)>0$ does not depend on $\beta$ but possibly on $T$ (this is guaranteed
by Lemma \ref{lem.infty}). 
If $\beta^2 \le 4\alpha/(\tau M(T)(1+\alpha^2))$ then
the free energy \eqref{1.E} fulfills the inequality
$$
  \frac{dE}{dt} + \frac{1}{2}\int_\Omega D\big(\rho_+|\na\log\rho_+ - \E|^2
	+ \rho_-|\na\log\rho_- - \E|^2\big)dx \le 0, \quad 0\le t\le T.
$$
\end{proposition}

In this proposition, the diffusion constant $D=D(x)$ is allowed to depend on $x$.

\begin{proof}
We denote the von-Neumann entropy part by $E_{\rm spin}$, 
the electromagnetic energy by $E_{\rm em}$, and the exchange energy by $E_{\rm ex}$.
By computing the time derivative of $E_{\rm spin}$ and employing 
\eqref{1.rho}-\eqref{1.s} and $2\rho=\rho_++\rho_-$, we find that
\begin{align*}
  \frac{dE_{\rm spin}}{dt}
	&= \frac12\int_\Omega D\bigg(\pa_t\rho_+\log\frac{\rho_+}{\rho_D}
	+ \pa_t\rho_-\log\frac{\rho_-}{\rho_D}\bigg)dx \\
	&= \frac12\int_\Omega D\bigg(\pa_t\rho\log\frac{\rho_+\rho_-}{\rho_D^2}
	+ \frac{\s}{|\s|}\cdot\pa_t\s\log\frac{\rho_+}{\rho_-}\bigg)dx \\
	&= -\frac12\int_\Omega D\na\log\frac{\rho_+\rho_-}{\rho_D^2}
	\cdot(\na\rho-\rho\E)dx 
	- \frac12\int_\Omega D\sum_{j=1}^3
	\frac{s_j}{|\s|}\na\log\frac{\rho_+}{\rho_-}
	\cdot(\na s_j-s_j\E)dx \\
	&\phantom{xx}{}- \frac12\int_\Omega D\sum_{j=1}^3
	\log\frac{\rho_+}{\rho_-}\na\bigg(\frac{s_j}{|\s|}\bigg)
	\cdot(\na s_j-s_j\E)dx 
	- \frac12\int_\Omega\frac{|\s|}{\tau}\log\frac{\rho_+}{\rho_-}dx \\
	&= I_1 + I_2 + I_3 + I_4.
\end{align*}
The second integral becomes
$$
  I_2	= -\frac12\int_\Omega D\na\log\frac{\rho_+}{\rho_-}\cdot(\na|\s|-|\s|\E)dx.
$$
Taking into account \eqref{nas}, 
we can reformulate a part of the integrand of $I_3$:
\begin{align*}
  \sum_{j=1}^3 & \na\bigg(\frac{s_j}{|\s|}\bigg)\cdot(\na s_j-s_j\E)
	= \frac{1}{|\s|}\sum_{j,k=1}^3\bigg(\delta_{jk} - \frac{s_js_k}{|\s|^2}\bigg)
  \na s_j\cdot\na s_k \\
	&{}
	- \frac{1}{|\s|}\sum_{j,k=1}^3\bigg(\delta_{jk} - \frac{s_js_k}{|\s|^2}\bigg)
	s_j\na s_k\cdot \E 
	= \frac{1}{|\s|}\sum_{j,k=1}^3\bigg(\delta_{jk} - \frac{s_js_k}{|\s|^2}\bigg)
  \na s_j\cdot\na s_k.
\end{align*}
The matrix $A=(a_{jk})$, defined by $a_{jk}=\frac12(\delta_{jk}-s_js_k/|\s|^2)$,
is a projection and satisfies $A^2=A$. Consequently,
\begin{align*}
  \sum_{j=1}^3\na\bigg(\frac{s_j}{|\s|}\bigg)\cdot(\na s_j-s_j\E)
	&= \frac{2}{|\s|}\sum_{i=1}^3\frac{\pa \s}{\pa x_i}A \frac{\pa \s}{\pa x_i}
	= \frac{2}{|\s|}\sum_{i=1}^3\frac{\pa \s}{\pa x_i}A^2 \frac{\pa \s}{\pa x_i} \\
	&= \frac{2}{|\s|}|A\na\s|^2 = 2|\s|\Big|\na\frac{\s}{|\s|}\Big|^2,
\end{align*}
and we infer that
$$
  I_3 = -\int_\Omega D\Big|\na\frac{\s}{|\s|}\Big|^2|\s|
	\log\frac{\rho_+}{\rho_-}dx.
$$
Then combining the integrals $I_1$ and $I_2$, we obtain
\begin{align*}
  \frac{dE_{\rm spin}}{dt}
	&= -\frac12\int_\Omega D\Big((\na\log\rho_+ - \na\log\rho_D)
	\cdot(\na\rho_+ - \rho_+\E) \\
	&\phantom{xxxxx}{}+ (\na\log\rho_- - \na\log\rho_D)
	\cdot(\na\rho_- - \rho_-\E)\Big)dx \\
	&\phantom{xx}{}
	- \int_\Omega\bigg(\frac{1}{2\tau}+\Big|\na\frac{\s}{|\s|}\Big|^2\bigg)
	|\s|\log\frac{\rho_+}{\rho_-}dx.
\end{align*}

Next, we compute the time derivatives of $E_{\rm em}$ and $E_{\rm ex}$:
\begin{align*}
  \frac{dE_{\rm em}}{dt}
	&= \int_\Omega D(\na\rho-\rho\E)\cdot(\E-\na\log\rho_D)
	- \int_\omega\H\cdot\pa_t\m dx \\
	&= \frac12\int_\Omega D\Big((\na\rho_+ - \rho_+\E)\cdot(\E-\na\log\rho_D) \\
	&\phantom{xx}{}
	+ (\na\rho_- - \rho_-\E)\cdot(\E-\na\log\rho_D)\Big)dx 
	- \int_\omega\H\cdot\pa_t\m dx, \\
	\frac{dE_{\rm ex}}{dt} 
	&= \int_\omega\na\m\cdot\na\pa_t\m dx dx = -\int_\omega\Delta\m\cdot\pa_t\m dx.
\end{align*}
Adding all time derivatives, the terms involving $\na\log\rho_D$ cancel
and we end up with
\begin{align}
  \frac{dE}{dt}
	&= -\frac12\int_\Omega D\Big(\rho_+|\na\log\rho_+ - \E|^2
	+ \rho_-|\na\log\rho_- - \E|^2\Big)dx \label{en.aux1} \\
	&\phantom{xx}{}
	- \int_\Omega\bigg(\frac{1}{2\tau}+\left|\na\frac{\s}{|\s|}\right|^2\bigg)
	|\s|\log\frac{\rho_+}{\rho_-}dx
	- \int_\omega(\H+\Delta\m)\cdot\pa_t\m dx. \nonumber
\end{align}
We employ the LLG equation to reformulate the last integral:
\begin{align}
  -\int_\omega&(\H+\Delta\m)\cdot\pa_t\m dx \label{en.aux2} \\
	&= -\int_\omega\Big((\H+\Delta\m)\cdot(\m\times\beta\s)
	+ \alpha|(\H+\Delta\m)\times\m|^2 \nonumber \\
	&\phantom{xx}{}- \alpha((\H+\Delta\m)\times\m)\cdot(\m\times\beta\s)\Big)dx 
	\nonumber \\
	&= -\alpha\int_\omega|(\H+\Delta\m)\times\m|^2 dx 
	- \beta\int_\omega((\H+\Delta\m)\times\m)\cdot(\s-\alpha\m\times\s)dx.
	\nonumber
\end{align}

At this point, we need to make some estimates. Applying Young's inequality
to the last integral, it follows that
\begin{align*}
  -\beta\int_\omega & ((\H+\Delta\m)\times\m)\cdot(\s-\alpha\m\times\s)\Big)dx \\
	&\le \alpha\int_\omega |(\H+\Delta\m)\times\m|^2 dx 
	+ \frac{\beta^2}{4\alpha}\int_\omega|\s-\alpha\m\times\s|^2 dx \\
	&= \alpha\int_\omega |(\H+\Delta\m)\times\m|^2 dx 
	+ \frac{\beta^2}{4\alpha}\int_\omega\big(|\s|^2 + \alpha^2|\m|^2|\s|^2\big)dx \\
	&= \alpha\int_\omega |(\H+\Delta\m)\times\m|^2 dx 
	+ \beta^2\frac{1+\alpha^2}{4\alpha}\int_\omega|\s|^2 dx.
\end{align*}
Thus, \eqref{en.aux2} becomes
$$
  -\int_\omega(\H+\Delta\m)\cdot\pa_t\m dx
	\le \beta^2\frac{1+\alpha^2}{4\alpha}\int_\omega|\s|^2 dx.
$$
Since $\log((1+z)/(1-z))\ge 2z$ for $0<z<1$, we estimate
$$
  I_4 = -\int_\Omega \frac{|\s|}{2\tau}\log\frac{\rho_+}{\rho_-}dx
	= -\int_\Omega \frac{|\s|}{2\tau}\log\frac{1+|\s|/\rho}{1-|\s|/\rho}dx
	\le -\int_\Omega\frac{|\s|^2}{\tau\rho}dx.
$$
Inserting these estimates into \eqref{en.aux1}, we arrive at
\begin{align*}
  \frac{dE}{dt}
	&+ \frac12\int_\Omega D\Big(\rho_+|\na\log\rho_+ - \E|^2
	+ \rho_-|\na\log\rho_- - \E|^2\Big)dx \\
	&\le \int_\omega\bigg(\beta^2\frac{1+\alpha^2}{4\alpha} - \frac{1}{\tau\rho}
	\bigg)|\s|^2 dx.
\end{align*}
Since $\rho\le M(T)$, the result follows.
\end{proof}


\end{document}